\documentclass[11pt,a4paper]{article}

\usepackage[utf8]{inputenc}
\usepackage{amsmath,amssymb,amsthm,mathtools}
\usepackage{fullpage}
\usepackage{hyperref}
\usepackage{cite}

\theoremstyle{plain}
\newtheorem{theorem}{Theorem}[section]
\newtheorem{lemma}[theorem]{Lemma}
\newtheorem{proposition}[theorem]{Proposition}
\newtheorem{corollary}[theorem]{Corollary}

\theoremstyle{definition}
\newtheorem{remark}[theorem]{Remark}
\newtheorem{example}[theorem]{Example}

\newcommand{\R}{\mathbb{R}}
\newcommand{\C}{\mathbb{C}}
\newcommand{\spec}{\sigma}
\newcommand{\Reach}{\operatorname{Reach}}
\newcommand{\ppt}{\operatorname{ppt}}

\title{\textbf{The De Cock--De Moor Conjecture:\\ Proof, Normal Form, and Indefinite Angles}}

\author{Jeffrey Humpherys\\Department of Mathematics and Statistics\\Missouri University of Science and Technology\\Rolla, Missouri}
\date{\today}

\begin{document}
\maketitle

\begin{abstract}
De Cock and De Moor conjectured that two matrices built from the solutions of three coupled Lyapunov equations always share their eigenvalues, an identity linking principal angles between model subspaces with the canonical correlations between the past and future of a stochastic process. The conjecture was recorded as Problem 9.1 of \emph{Unsolved Problems in Mathematical Systems and Control Theory}. We prove it under the natural nonresonance condition, without stability or diagonalizability, and show that the two matrices are similar. The proof exposes a Cauchy-matrix structure created by the rank-one perturbation and reduces the problem to rational interpolation, and the main theorem has also been verified in the Lean~4 proof assistant. The same structure gives a normal form that depends only on the poles and zeros of the model. Off the stable set the eigenvalues become squared cosines of angles in an indefinite metric. We prove an inertia theorem for them and show that reflecting a pole or a zero across the unit circle acts on the angles by a principal pivot transform. For several inputs and outputs, where the conjecture fails, we identify the subspaces whose angles do give the canonical correlations. We also relate our approach to the independent proof of Gillberg and L\"ofberg.
\end{abstract}



\section{Introduction}

A basic question in stochastic subspace identification is how the geometry of model subspaces reflects the statistical dependence between the past and future of a stochastic process. Principal angles provide a geometric measure of the separation between subspaces, while canonical correlations measure the dependence between collections of random variables. These ideas have classical origins in the work of Jordan and Hotelling \cite{Jordan1875,Hotelling1936}, but they come together naturally in linear stochastic models. Subspace identification recovers a system from structured subspaces built from observed data, whereas the canonical correlations between past and future measure the memory carried by the process and, in the Gaussian case, determine their mutual information \cite{GelfandYaglom1959,JewellBloomfield1983}. De Cock and De Moor observed that these two viewpoints should be linked by a remarkably simple matrix identity. Establishing that identity would show that the nonzero past--future canonical correlations of a scalar autoregressive moving-average model are exactly the sines of certain principal angles associated with the model and its whitening inverse, and would also yield short proofs of identities underlying cepstral distance measures for comparing and clustering time series \cite{Martin2000,DeCock2002SCL,Boets2005}. They formulated this connection as a conjecture in the 2002 MTNS Problem Book and later as Problem 9.1 of \emph{Unsolved Problems in Mathematical Systems and Control Theory}, deliberately isolating it as a finite-dimensional matrix-algebraic problem \cite{DeCock2002MTNS,DeCock2004}. While this manuscript was being prepared, Gillberg and L\"ofberg independently posted another direct finite-dimensional proof based on a Lyapunov-kernel identity and the classical fact that two products taken in opposite orders share their nonzero eigenvalues \cite{GillbergLofberg2026}.

The truth of the spectral statement is not entirely without support in the literature. In the stable, minimum-phase setting of its system-theoretic interpretation, De Cock and De Moor remark in \cite{DeCock2004} that the identity can be established ``via a large detour'' through the canonical-correlation results of \cite[Theorems~3.11 and~3.12]{DeCock2002PhD}, which rest on stochastic arguments and a limiting autoregressive approximation. Scherrer treated the associated canonical-correlation relation in the stable setting through minimum-phase balanced truncation \cite{Scherrer2002}. In 2005, Boets, De Cock, Espinoza, and De Moor still described a direct proof as ``a challenge to the reader'' in connection with a related principal-angle identity \cite{Boets2005}. The proof given here was developed independently and proceeds through Cauchy matrices and rational interpolation. It requires neither stability nor minimum phase nor diagonalizability.

The precise statement is as follows. Given a real matrix $A\in\R^{n\times n}$ and vectors $v,w\in\R^n$, set $B=A+vw^T$. The original formulation is built on the single block Stein (discrete-time Lyapunov) equation
\begin{equation}
\label{eq-blockstein}
 \begin{pmatrix}P&R\\ R^T&Q\end{pmatrix}
 -\begin{pmatrix}A&0\\0&B^T\end{pmatrix}
 \begin{pmatrix}P&R\\ R^T&Q\end{pmatrix}
 \begin{pmatrix}A^T&0\\0&B\end{pmatrix}
 =\begin{pmatrix}v v^T & v w^T\\w v^T & w w^T\end{pmatrix}.
\end{equation}
Under a nonresonance hypothesis on the spectra of $A$ and $B$, made precise in Section~\ref{sec-results}, this equation has a unique solution, and uniqueness forces the solution to be symmetric. The block equation reduces to the three Stein equations
\begin{equation}
\label{eq-stein}
 P-APA^T=vv^T,
 \qquad
 Q-B^TQB=ww^T,
 \qquad
 R-ARB=vw^T.
\end{equation}
The conjecture asserts that if $P$, $Q$, and $I+PQ$ are nonsingular, then the matrices $P^{-1}RQ^{-1}R^T$ and $(I+PQ)^{-1}$ have the same eigenvalues.

The conjecture is striking because the two matrices have rather different forms. One is built from the cross solution $R$ together with both $P$ and $Q$, while the other depends only on the product $PQ$. A direct proof should therefore identify a common algebraic structure rather than attempt to force a pointwise equality between the two matrices.

Our proof proceeds in two stages. First, on a generic set where the spectra of $A$ and $A+vw^T$ are simple and nonzero, we diagonalize the two matrices independently. The blocks of the Stein solution become Cauchy matrices compressed by diagonal factors, and the rank-one structure enters through a factorization of the change-of-basis matrix between the two eigenbases. A single rational interpolation identity then yields two matrix identities, which turn out to be basis-free. Second, the two identities pass to all nonresonant data by density and continuity, and a short algebraic argument shows that they force the two matrices of the conjecture to be similar.

The same Cauchy structure answers more than the conjecture asks. In his thesis De Cock asked for an interpretation of the identity and its consequences \cite[p.~255]{DeCock2002PhD}. Off the stable set the identity loses its geometric meaning in the usual sense, because the Gramians are no longer positive definite and the common eigenvalues need not lie in the unit interval. We show that the common eigenvalues are always of the form $1-\lambda^2$, where $\lambda$ runs over the eigenvalues of an explicit matrix built from the two characteristic polynomials alone. In the stable case these $\lambda$ are the canonical correlations with a sign attached, and the signs are counted by a Cauchy index. Off the stable set we prove an inertia theorem. Real eigenvalues never exceed one, negative eigenvalues are forced when the numbers of poles and zeros outside the unit circle differ, and complex eigenvalues can occur only when both straddle the circle. We then show how the unstable case is tied to the stable one. Moving a root across the unit circle, together with its conjugate when it is not real, changes the spectral density of the model only by a constant factor, and it changes the signed sines by a principal pivot transform at that root. Finally we treat models with several inputs and outputs, where the conjecture is known to fail \cite{DeCock2004}. We explain the failure, identify the correct pair of subspaces, and prove an interlacing inequality.

The appearance of Cauchy matrices is closely related to the classical algebraic structure of Lyapunov and Stein equations in companion form; see, for example, \cite{HeinigJungnickel1986}. Cauchy matrices of the same form already appear in De Cock's computation of the principal angles between two autoregressive models from their poles \cite[Proposition~3.9]{DeCock2002PhD}.

This paper is organized as follows. Section~\ref{sec-results} states the hypotheses and the main results. Section~\ref{sec-cauchy} reduces the Stein solutions to Cauchy matrices on the simple-spectrum set and derives the two basis-free identities there, and Section~\ref{sec-proof} completes the proof of the conjecture. Section~\ref{sec-normal} develops the normal form and the signs of the sines, Section~\ref{sec-inertia} proves the inertia theorem, and Section~\ref{sec-reflection} proves the reflection principle. Section~\ref{sec-mimo} treats several channels, and Section~\ref{sec-discussion} concludes with a discussion. The supplementary material contains a formal verification of the main theorem in the Lean~4 proof assistant with the Mathlib library, which proves the two basis-free identities directly rather than through the density argument, together with a script that reproduces the numerical checks mentioned in the text.

\section{Statement of results}\label{sec-results}

For a real matrix $A\in\R^{n\times n}$ and vectors $v,w\in\R^n$, set $B=A+vw^T$. The \emph{nonresonance hypothesis} states that for $\alpha_i\in\spec(A)$ and $\beta_j\in\spec(B)$, none of the products $\alpha_i\alpha_j$, $\beta_i\beta_j$, or $\alpha_i\beta_j$ is equal to $1$. Under this hypothesis each equation in \eqref{eq-stein} is uniquely solvable, giving matrices $P,Q,R\in\R^{n\times n}$. Indeed, after vectorization their coefficient matrices are respectively
\[
 I-A\otimes A,\qquad I-B^T\otimes B^T,\qquad I-B^T\otimes A,
\]
and their eigenvalues are $1-\alpha_i\alpha_j$, $1-\beta_i\beta_j$, and $1-\alpha_i\beta_j$. Uniqueness also implies that $P=P^T$ and $Q=Q^T$, since transposition leaves their respective Stein equations unchanged. Since the data are real, the hypothesis also excludes eigenvalues on the unit circle, because such an eigenvalue $\alpha$ would come with $\bar\alpha=1/\alpha$.

The first result says that a shared eigenvalue of $A$ and $B$ is impossible on the domain of the conjecture.

\begin{proposition}
\label{prop:disjoint}
Assume the nonresonance hypothesis. If $P$ and $Q$ are nonsingular, then $\spec(A)\cap\spec(B)=\varnothing$. Moreover no left eigenvector of $A$ is orthogonal to $v$ and no right eigenvector of $B$ is orthogonal to $w$, so the pairs $(A,v)$ and $(B^T,w)$ are controllable.
\end{proposition}

\begin{proof}
We work over $\C$ with the plain transpose. First, no left eigenvector of $A$ is orthogonal to $v$. Suppose $\ell\neq0$, $\ell^TA=\alpha\ell^T$, and $\ell^Tv=0$. Multiplying the first equation of \eqref{eq-stein} on the left by $\ell^T$ gives
\[
 \ell^TP-\alpha\,\ell^TPA^T=(\ell^Tv)\,v^T=0,
\]
that is, $\ell^TP(I-\alpha A^T)=0$. The eigenvalues of $I-\alpha A^T$ are $1-\alpha a_j$ with $a_j\in\spec(A)$, and each is nonzero because the nonresonance hypothesis excludes $\alpha a_j=1$. Hence $\ell^TP=0$, contradicting the nonsingularity of $P$.

Second, no right eigenvector of $B$ is orthogonal to $w$. Suppose $r\neq0$, $Br=\lambda r$, and $w^Tr=0$. Multiplying the second equation of \eqref{eq-stein} on the right by $r$ gives
\[
 Qr-\lambda B^TQr=w\,(w^Tr)=0,
\]
that is, $(I-\lambda B^T)Qr=0$. The eigenvalues of $I-\lambda B^T$ are $1-\lambda b_j$ with $b_j\in\spec(B)$, and each is nonzero by nonresonance. Hence $Qr=0$, contradicting the nonsingularity of $Q$.

Finally, suppose $\lambda\in\spec(A)\cap\spec(B)$. Choose a left eigenvector $\ell$ of $A$ and a right eigenvector $r$ of $B$, both for the eigenvalue $\lambda$. Then
\[
 (\ell^Tv)(w^Tr)=\ell^T(vw^T)r=\ell^T(B-A)r=\lambda\,\ell^Tr-\lambda\,\ell^Tr=0,
\]
so $\ell^Tv=0$ or $w^Tr=0$, contradicting one of the two preceding observations. The controllability statements are the Popov--Belevitch--Hautus test applied to the first two observations.
\end{proof}

Write $p_A(z)=\det(zI-A)$ and $q_A(z)=\det(I-zA)=z^np_A(1/z)$ for the characteristic polynomial of $A$ and its reversal, define $p_B$ and $q_B$ in the same way, and set
\[
 \phi_A=\frac{p_A}{q_A},\qquad \phi_B=\frac{p_B}{q_B}.
\]
The eigenvalues of $q_B(A)$ are the numbers $q_B(a_i)=\prod_j(1-b_ja_i)$, which are nonzero by nonresonance, and the eigenvalues of $p_B(A)$ are nonzero by Proposition~\ref{prop:disjoint}. Hence the real matrix
\begin{equation}\label{eq-F}
 F=\phi_B(A)=p_B(A)\,q_B(A)^{-1}
\end{equation}
is defined and nonsingular on the domain of the conjecture. The main theorem is the following.

\begin{theorem}[De Cock--De Moor conjecture]\label{thm-main}
Under the nonresonance hypothesis, suppose that $P$ and $Q$ are nonsingular. Then
\begin{equation}\label{eq-I1I2}
 RQ^{-1}R^T=P-FPF^T
 \qquad\text{and}\qquad
 F\bigl(P+Q^{-1}\bigr)F^T=P .
\end{equation}
Consequently $R$ and $I+PQ$ are nonsingular, and the real matrix $U=P^{-1}F$ satisfies
\begin{equation}
\label{eq-main-similarity}
U^{-1}\bigl(P^{-1}RQ^{-1}R^T\bigr)U=(I+PQ)^{-1}.
\end{equation}
In particular
\begin{equation}
\label{eq-main-spectrum}
\spec\!\left(P^{-1}RQ^{-1}R^T\right) = \spec\!\left((I+PQ)^{-1}\right),
\end{equation}
with algebraic multiplicities, so the two matrices have the same characteristic polynomial.
\end{theorem}

Our formulation removes the separate assumption of the conjecture that $I+PQ$ be nonsingular, and it strengthens equality of spectra to similarity. The matrix $U^{-1}=F^{-1}P$ is the intertwiner constructed by Gillberg and L\"ofberg \cite{GillbergLofberg2026}. In our argument the two identities \eqref{eq-I1I2} come from the Cauchy structure, and the intertwiner is read off from them. We recognized the identities of Section~\ref{sec-cauchy} as coordinate forms of \eqref{eq-I1I2} in comparing our argument with their construction. The relation between the two approaches is discussed in Section~\ref{sec-discussion}.

The next result says that the common spectrum depends only on the two characteristic polynomials, and it gives a matrix of size $n$ from which it can be computed. Let $Z_A$ be the companion matrix of $p_A$, that is, the matrix of multiplication by $z$ modulo $p_A$ in the basis $1,z,\dots,z^{n-1}$, and let $J$ be the $n\times n$ flip matrix. The matrix $q_B(Z_A)$ is nonsingular by nonresonance, so
\begin{equation}\label{eq-T}
 T=J\,\phi_B(Z_A)=J\,p_B(Z_A)\,q_B(Z_A)^{-1}
\end{equation}
is defined. It depends only on $p_A$ and $p_B$, and it is nonsingular by Proposition~\ref{prop:disjoint}.

\begin{theorem}[Normal form]\label{thm-normal}
Under the hypotheses of Theorem~\ref{thm-main}, the two matrices in \eqref{eq-main-spectrum} have the characteristic polynomial of
\[
 I-T^2 .
\]
Thus their common eigenvalues are the numbers $1-\lambda^2$ with $\lambda\in\spec(T)$, counted with multiplicity. If $A$ and $B$ are Schur stable, the eigenvalues of $T$ are real, and $\#\{\lambda>0\}-\#\{\lambda<0\}$ equals the Cauchy index of $p_B/p_A$ on the real line.
\end{theorem}

In the stable setting the eigenvalues of $(I+PQ)^{-1}$ are $1-\rho_k^2$, where $\rho_1,\dots,\rho_n$ are the nonzero canonical correlations between the past and future of the stationary process generated by the model \cite{DeCock2002PhD,LindquistPicci2015}. Theorem~\ref{thm-normal} therefore says that the canonical correlations are the absolute values of the eigenvalues of $T$, and the eigenvalues themselves carry signs whose count is a classical Cauchy index. We call the eigenvalues of $T$ the \emph{signed sines}.

Off the stable set the eigenvalues in \eqref{eq-main-spectrum} need not lie in the unit interval. Let $\kappa_A$ and $\kappa_B$ denote the numbers of eigenvalues of $A$ and $B$, counted with multiplicity, of modulus greater than one.

\begin{theorem}[Inertia]\label{thm-inertia}
Under the hypotheses of Theorem~\ref{thm-main}, let $N_-$ be the number of negative eigenvalues of the two matrices in \eqref{eq-main-spectrum}, $N_+$ the number of real eigenvalues in $(0,1)$, and $N_c$ the number of conjugate pairs of nonreal eigenvalues, all counted with algebraic multiplicity. Then
\begin{enumerate}
\item[(a)] every real eigenvalue lies in $(-\infty,0)\cup(0,1)$, so $N_-+N_++2N_c=n$;
\item[(b)] $N_c\le\min\{\kappa_A,\,n-\kappa_A,\,\kappa_B,\,n-\kappa_B\}$;
\item[(c)] $N_-\ge|\kappa_A-\kappa_B|$ and $N_-\equiv\kappa_A-\kappa_B\pmod 2$;
\item[(d)] $N_+\ge|n-\kappa_A-\kappa_B|$ and $N_+\equiv n-\kappa_A-\kappa_B\pmod 2$.
\end{enumerate}
In particular, if all eigenvalues of $A$, or all eigenvalues of $B$, lie on one side of the unit circle, then every eigenvalue is real.
\end{theorem}

Part (a) is the main point. It is not a consequence of general indefinite linear algebra, since a product of two indefinite symmetric matrices can have real eigenvalues anywhere. It says that no angle is more parallel than parallel.

Two further results connect the unstable and stable cases and treat several channels. Theorem~\ref{thm-reflection} in Section~\ref{sec-reflection} shows that reflecting a pole or a zero across the unit circle replaces the signed sines by those of a principal pivot transform, and Corollary~\ref{cor-allreflect} makes this explicit when all poles lie outside the circle and all zeros inside. Section~\ref{sec-mimo} identifies the subspaces whose angles give the canonical correlations of a model with several channels and proves the interlacing inequality of Theorem~\ref{thm-interlace}.

\section{Cauchy structure on the simple-spectrum set}\label{sec-cauchy}

In this section we make the following assumptions in addition to the nonresonance hypothesis.
\begin{itemize}
 \item[(S1)] $A$ and $B$ each have $n$ distinct eigenvalues, all nonzero;
 \item[(S2)] $P$ and $Q$ are nonsingular.
\end{itemize}
The eigenvalues may be complex. All transposes below are plain transposes, also for complex matrices; no complex conjugation occurs anywhere, and all identities are bilinear.

By (S1) we may fix diagonalizations
\begin{equation}\label{eq-diag}
 A=S\Lambda S^{-1},
 \qquad
 B=S_B\Gamma S_B^{-1},
 \qquad
 \Lambda=\operatorname{diag}(a_1,\ldots,a_n),
 \quad
 \Gamma=\operatorname{diag}(b_1,\ldots,b_n),
\end{equation}
with $S,S_B\in\C^{n\times n}$ nonsingular. Set
\begin{equation}\label{eq-xy}
 x=S^{-1}v,\qquad y=S_B^Tw,
 \qquad
 D_x=\operatorname{diag}(x_1,\ldots,x_n),
 \quad
 D_y=\operatorname{diag}(y_1,\ldots,y_n),
\end{equation}
and transform the Stein solutions two-sidedly,
\begin{equation}\label{eq-hats}
 \widehat P=S^{-1}PS^{-T},
 \qquad
 \widehat Q=S_B^TQS_B,
 \qquad
 \widehat R=S^{-1}RS_B.
\end{equation}
Substituting \eqref{eq-diag} into \eqref{eq-stein} and using $\Lambda^T=\Lambda$ and $\Gamma^T=\Gamma$ gives
\begin{equation}\label{eq-hatstein}
 \widehat P-\Lambda\widehat P\Lambda=xx^T,
 \qquad
 \widehat Q-\Gamma\widehat Q\Gamma=yy^T,
 \qquad
 \widehat R-\Lambda\widehat R\Gamma=xy^T.
\end{equation}
Define the Cauchy matrices
\begin{equation}\label{eq-CauchyMatrices}
 G_A=\Bigl[\frac{1}{1-a_ia_j}\Bigr]_{i,j=1}^n,
 \qquad
 G_B=\Bigl[\frac{1}{1-b_ib_j}\Bigr]_{i,j=1}^n,
 \qquad
 H=\Bigl[\frac{1}{1-a_ib_j}\Bigr]_{i,j=1}^n,
\end{equation}
whose denominators are nonzero by nonresonance. Their nonsingularity follows from the generalized Cauchy determinant
\begin{equation}\label{eq-general-cauchy-det}
 \det\Bigl[\frac{1}{1-s_it_j}\Bigr]_{i,j=1}^n
 =\frac{\prod_{i<j}(s_i-s_j)(t_i-t_j)}{\prod_{i,j}(1-s_it_j)},
\end{equation}
valid for complex parameters with the displayed ordering of the Vandermonde factors. The choices $(s,t)=(a,a)$, $(b,b)$, and $(a,b)$, together with (S1), show that $G_A$, $G_B$, and $H$ are nonsingular.

\begin{proposition}
\label{prop-factorizations}
Under (S1),
\begin{align}
 \widehat P&=D_xG_AD_x, \label{eq-Pfactor}\\
 \widehat Q&=D_yG_BD_y, \label{eq-Qfactor}\\
 \widehat R&=D_xHD_y. \label{eq-Rfactor}
\end{align}
If in addition (S2) holds, then every entry of $x$ and of $y$ is nonzero, so $D_x$ and $D_y$ are nonsingular.
\end{proposition}

\begin{proof}
The equations \eqref{eq-hatstein} are diagonal and solve entrywise,
$\widehat P_{ij}(1-a_ia_j)=x_ix_j$, $\widehat Q_{ij}(1-b_ib_j)=y_iy_j$, and $\widehat R_{ij}(1-a_ib_j)=x_iy_j$, which are \eqref{eq-Pfactor}--\eqref{eq-Rfactor}. If some $x_i$ were zero, the $i$th row of $\widehat P=S^{-1}PS^{-T}$ would vanish, contradicting (S2); likewise for $y_j$ and $\widehat Q$.
\end{proof}

The rank-one structure of the perturbation enters through the change of basis between the two eigenbases, and only there.

\begin{lemma}
\label{lem-changeofbasis}
Assume (S1)--(S2) and let $Z=S^{-1}S_B$. Then
\begin{equation}\label{eq-Cdef}
 C=\Bigl[\frac{1}{b_j-a_i}\Bigr]_{i,j=1}^n
\end{equation}
is well defined, $Z=D_xCD_y$, and $C$ is nonsingular.
\end{lemma}

\begin{proof}
The spectra of $A$ and $B$ are disjoint by Proposition~\ref{prop:disjoint}, so $C$ is well defined. From $B-A=vw^T$,
\[
 Z\Gamma-\Lambda Z
 =S^{-1}BS_B-S^{-1}AS_B
 =S^{-1}vw^TS_B
 =xy^T,
\]
so $(b_j-a_i)Z_{ij}=x_iy_j$, that is $Z_{ij}=x_iC_{ij}y_j$. Since $Z$, $D_x$, and $D_y$ are nonsingular, so is $C$.
\end{proof}

We now come to the interpolation identity. Let $\beta_1,\ldots,\beta_n\in\C$ be distinct with $\beta_i\beta_j\neq1$ for all $i,j$. Define
\begin{equation}\label{eq-phi}
 \phi_\beta(z)=\prod_{r=1}^n\frac{z-\beta_r}{1-\beta_rz},
\end{equation}
the vector-valued rational function
\begin{equation}\label{eq-kell}
 k_\beta(z)=
 \begin{pmatrix}
 (1-\beta_1z)^{-1}\\ \vdots\\ (1-\beta_nz)^{-1}
 \end{pmatrix},
\end{equation}
and the Cauchy matrix $G_\beta=[(1-\beta_i\beta_j)^{-1}]_{i,j=1}^n$, which is nonsingular by \eqref{eq-general-cauchy-det}. For the eigenvalues of $B$, $\phi_\beta$ is the function $\phi_B$ of Section~\ref{sec-results}.

\begin{lemma}[Interpolation identity]
\label{lem-cauchy-identity}
As an identity of rational functions in $z$ and $\zeta$,
\begin{equation}\label{eq-kernel1}
 k_\beta(z)^TG_\beta^{-1}k_\beta(\zeta)
 =\frac{1-\phi_\beta(z)\phi_\beta(\zeta)}{1-z\zeta}.
\end{equation}
\end{lemma}

\begin{proof}
Write $p(z)=\prod_r(z-\beta_r)$ and $q(z)=\prod_r(1-\beta_rz)$, so that $\phi_\beta=p/q$. For $\zeta\neq0$,
\begin{equation}\label{eq-reciprocalphi}
 \phi_\beta(1/\zeta)=\phi_\beta(\zeta)^{-1},
\end{equation}
since $p(1/\zeta)=q(\zeta)/\zeta^n$ and $q(1/\zeta)=p(\zeta)/\zeta^n$. Fix $\zeta$ outside the finite set where the expressions below are undefined; in particular $\zeta\neq0$ and $p(\zeta)q(\zeta)\neq0$. Let
\[
 f_\zeta(z)=\frac{1-\phi_\beta(z)\phi_\beta(\zeta)}{1-z\zeta}
 =\frac{q(z)-\phi_\beta(\zeta)p(z)}{q(z)(1-z\zeta)}.
\]
By \eqref{eq-reciprocalphi} the numerator $q(z)-\phi_\beta(\zeta)p(z)$ vanishes at $z=1/\zeta$, hence is divisible by $1-z\zeta$, and
\[
 r_\zeta(z)=\frac{q(z)-\phi_\beta(\zeta)p(z)}{1-z\zeta}
\]
is a polynomial of degree at most $n-1$; thus $q(z)f_\zeta(z)=r_\zeta(z)$. Likewise, with $g_\zeta(z)=k_\beta(z)^TG_\beta^{-1}k_\beta(\zeta)$, the product $q(z)g_\zeta(z)$ is a polynomial of degree at most $n-1$, because $q(z)(1-\beta_iz)^{-1}=\prod_{r\neq i}(1-\beta_rz)$. At each of the $n$ distinct points $z=\beta_j$, the hypothesis gives $q(\beta_j)\neq0$; moreover $q(\zeta)\neq0$ ensures $1-\beta_j\zeta\neq0$, so that, using $\phi_\beta(\beta_j)=0$,
\[
 f_\zeta(\beta_j)=\frac{1}{1-\beta_j\zeta}.
\]
On the other hand, $k_\beta(\beta_j)^T$ is the $j$th row of $G_\beta$, so
\[
 g_\zeta(\beta_j)=e_j^Tk_\beta(\zeta)=\frac{1}{1-\beta_j\zeta}
\]
as well. The polynomial $q(z)\bigl(f_\zeta(z)-g_\zeta(z)\bigr)$ has degree at most $n-1$ and vanishes at the $n$ distinct points $\beta_1,\ldots,\beta_n$; it is therefore identically zero. Hence \eqref{eq-kernel1} holds, for each such $\zeta$, as an identity of rational functions of $z$, and then, both sides being rational in $(z,\zeta)$, as an identity of rational functions of $(z,\zeta)$.
\end{proof}

\begin{remark}
\label{rem-modelspace}
Lemma \ref{lem-cauchy-identity} is classical in substance. When the nodes are real and lie in $(-1,1)$, $\phi_\beta$ is a finite Blaschke product. Upon setting $\zeta=\overline{\xi}$, the right-hand side of \eqref{eq-kernel1} becomes the reproducing kernel of the model space $K_{\phi_\beta}=H^2\ominus\phi_\beta H^2$; equivalently, the displayed bilinear expression is its polarization. The matrix $G_\beta$ is the Gram matrix of the Szeg\H{o} kernels $z\mapsto(1-\beta_iz)^{-1}$, which span $K_{\phi_\beta}$, and \eqref{eq-kernel1} is the corresponding change-of-basis formula; see, e.g., \cite{Nikolskii1986,FricainMashreghi2016}. For general complex nodes and the plain transpose, \eqref{eq-kernel1} is a bilinear algebraic continuation of this identity, proved above directly by rational interpolation. Related Gram-matrix descriptions of Stein and Lyapunov solutions appear in \cite{Ptak1983,HeinigJungnickel1986}, and \cite{Ferrante2026} is a recent use of the Cauchy structure of Lyapunov solutions.
\end{remark}

Under the standing assumptions (S1)--(S2), define
\begin{equation}\label{eq-Phi}
 \Phi=\operatorname{diag}\bigl(\phi_B(a_1),\ldots,\phi_B(a_n)\bigr).
\end{equation}
By Proposition~\ref{prop:disjoint} and nonresonance, every diagonal entry of $\Phi$ is finite and nonzero, so $\Phi$ is nonsingular.

\begin{corollary}
\label{cor-two-identities}
Under (S1) and (S2), define
\begin{equation}\label{eq-MN}
 M=HG_B^{-1}H^T,
 \qquad
 N=CG_B^{-1}C^T.
\end{equation}
Then
\begin{align}
 M&=G_A-\Phi G_A\Phi, \label{eq-idH}\\
 N&=\Phi^{-1}G_A\Phi^{-1}-G_A. \label{eq-idC}
\end{align}
\end{corollary}

\begin{proof}
The $i$th row of $H$ is $k_b(a_i)^T$, where $k_b$ denotes \eqref{eq-kell} for the node family $b$. Lemma \ref{lem-cauchy-identity} with $\beta=b$, evaluated at $(z,\zeta)=(a_i,a_j)$, which is admissible because $1-b_ra_i\neq0$ and $1-a_ia_j\neq0$, gives
\[
 \bigl(HG_B^{-1}H^T\bigr)_{ij}
 =\frac{1-\phi_B(a_i)\phi_B(a_j)}{1-a_ia_j},
\]
which is the $(i,j)$ entry of $G_A-\Phi G_A\Phi$; this proves \eqref{eq-idH}.

For \eqref{eq-idC}, apply the lemma to the reciprocal family $\beta=(b_1^{-1},\ldots,b_n^{-1})$, which is admissible because its entries are distinct and nonzero by (S1), and $b_i^{-1}b_j^{-1}\neq1$ is equivalent to $b_ib_j\neq1$. Recall that $\Gamma=\operatorname{diag}(b_1,\ldots,b_n)$. For this family,
\[
 \bigl(k_{b^{-1}}(z)\bigr)_r=\frac{1}{1-z/b_r}=\frac{b_r}{b_r-z},
 \qquad
 G_{b^{-1}}=\Bigl[\frac{b_ib_j}{b_ib_j-1}\Bigr]_{i,j=1}^n=-\Gamma G_B\Gamma,
 \qquad
 \phi_{b^{-1}}=\phi_B^{-1},
\]
the last identity because $(z-b_r^{-1})/(1-zb_r^{-1})=(1-zb_r)/(z-b_r)$ for each $r$. We evaluate at $(z,\zeta)=(a_i,a_j)$, which is admissible because $a_i\neq b_r$. The row vectors $k_{b^{-1}}(a_i)^T$ assemble into the matrix $C\Gamma$, and \eqref{eq-kernel1} for the reciprocal family reads
\[
 (C\Gamma)\bigl(-\Gamma G_B\Gamma\bigr)^{-1}(C\Gamma)^T=G_A-\Phi^{-1}G_A\Phi^{-1},
\]
which is \eqref{eq-idC}.
\end{proof}

The two identities of Corollary~\ref{cor-two-identities} are coordinate forms of the identities \eqref{eq-I1I2} of Theorem~\ref{thm-main}.

\begin{proposition}\label{prop-I1I2-generic}
Under (S1)--(S2), the identities \eqref{eq-I1I2} hold.
\end{proposition}

\begin{proof}
Since $A=S\Lambda S^{-1}$, we have $F=S\Phi S^{-1}$. From \eqref{eq-hats}, $P=S\widehat PS^T$, $Q^{-1}=S_B\widehat Q^{-1}S_B^T$, and $R=S\widehat RS_B^{-1}$. For the first identity,
\[
 RQ^{-1}R^T=S\,\widehat R\widehat Q^{-1}\widehat R^T S^T
 =S\,D_xHG_B^{-1}H^TD_x\,S^T
 =S\,D_x\bigl(G_A-\Phi G_A\Phi\bigr)D_x\,S^T
\]
by Proposition~\ref{prop-factorizations} and \eqref{eq-idH}. Since diagonal matrices commute, the right side is $S(\widehat P-\Phi\widehat P\Phi)S^T=P-FPF^T$. For the second identity, $S_B=SZ$ and Lemma~\ref{lem-changeofbasis} give
\[
 Q^{-1}=S\,Z\widehat Q^{-1}Z^T\,S^T
 =S\,D_xCD_y\,D_y^{-1}G_B^{-1}D_y^{-1}\,D_yC^TD_x\,S^T
 =S\,D_xND_x\,S^T,
\]
and \eqref{eq-idC} gives $D_xND_x=\Phi^{-1}\widehat P\Phi^{-1}-\widehat P$. Hence $Q^{-1}=F^{-1}PF^{-T}-P$, which is the second identity.
\end{proof}

\section{Proof of the main theorem}\label{sec-proof}

It remains to remove the simple, nonzero spectral assumptions used to expose the Cauchy structure. This requires only density and continuity.

Let $\mathcal A$ denote the set of real parameter triples $(A,v,w)$ satisfying the nonresonance hypothesis and for which $P$ and $Q$ are nonsingular. On the nonresonant set, the vectorized Stein equations are uniquely solvable linear systems whose coefficients depend polynomially on the data, so $P$, $Q$, and $R$ depend continuously on $(A,v,w)$. Thus $\mathcal A$ is open.

\begin{lemma}
\label{lem-density}
The triples in $\mathcal A$ for which $A$ and $B=A+vw^T$ have simple, nonzero spectra are dense in $\mathcal A$.
\end{lemma}

\begin{proof}
The failure of simple, nonzero spectra is detected by the polynomial
\[
 d(A,v,w)=\det(A)\det(B)\,\operatorname{disc}(p_A)\operatorname{disc}(p_B).
\]
This polynomial is not identically zero. For example, if $A=\operatorname{diag}(2,\ldots,n+1)/(n+2)$ and $v=0$, then $B=A$ and $d(A,v,w)\neq0$ for every $w$. Hence $\{d\neq0\}$ is dense in the full parameter space. Since $\mathcal A$ is open, $\mathcal A\cap\{d\neq0\}$ is dense in $\mathcal A$.
\end{proof}

\begin{proof}[Proof of Theorem \ref{thm-main}]
On $\mathcal A$ the matrices $P$, $Q^{-1}$, $R$, and $F=p_B(A)q_B(A)^{-1}$ depend continuously on the data, the last because $\det q_B(A)=\prod_{i,j}(1-b_ja_i)$ does not vanish. Both sides of each identity in \eqref{eq-I1I2} are therefore continuous on $\mathcal A$. They agree on a dense subset by Proposition~\ref{prop-I1I2-generic} and Lemma~\ref{lem-density}, hence everywhere on $\mathcal A$.

Now fix data in $\mathcal A$. The matrix $F$ is nonsingular by Proposition~\ref{prop:disjoint}. The second identity gives $P+Q^{-1}=F^{-1}PF^{-T}$, so
\[
 I+PQ=\bigl(P+Q^{-1}\bigr)Q=F^{-1}PF^{-T}Q
\]
is nonsingular, and
\[
 (I+PQ)^{-1}=Q^{-1}F^TP^{-1}F=\bigl(F^{-1}PF^{-T}-P\bigr)F^TP^{-1}F=I-PF^TP^{-1}F .
\]
The first identity gives
\[
 P^{-1}RQ^{-1}R^T=I-P^{-1}FPF^T .
\]
With $U=P^{-1}F$ these read $P^{-1}RQ^{-1}R^T=I-U\,(PF^T)$ and $(I+PQ)^{-1}=I-(PF^T)\,U$, and therefore
\[
 U^{-1}\bigl(P^{-1}RQ^{-1}R^T\bigr)U=I-(PF^T)U=(I+PQ)^{-1},
\]
which is \eqref{eq-main-similarity}. The right side is nonsingular, so $R$ is nonsingular as well.
\end{proof}

\begin{remark}[Complex data]\label{rem-complex}
The same proof applies to complex $A,v,w$ with the ordinary transpose throughout. The generic argument is already complex, and the corresponding generic spectral set is dense over $\C$. If one instead uses conjugate transposes, so that $P-APA^*=vv^*$ and $Q-B^*QB=ww^*$ with $B=A+vw^*$, the cross equation should be taken as $R-\bar ARB=\bar vw^*$, which makes $R$ the cross Gram matrix of the two subspaces for the Hermitian inner product. With that choice $\spec(\bar P^{-1}RQ^{-1}R^*)=\spec((I+PQ)^{-1})$ in numerical experiments, while the choice $R-ARB=vw^*$ fails in generic examples.
\end{remark}

\section{The normal form and the signs of the sines}\label{sec-normal}

We return to the simple-spectrum set and to the notation of Section~\ref{sec-cauchy}. One residue computation supplies everything we need about the Cauchy matrices. Put
\[
 D_A=\operatorname{diag}\Bigl(\frac{q_A(a_1)}{p_A'(a_1)},\dots,\frac{q_A(a_n)}{p_A'(a_n)}\Bigr),
\]
and define $D_B$ in the same way from the eigenvalues of $B$.

\begin{lemma}\label{lem-residue}
Let $a_1,\dots,a_n$ be distinct and nonzero with $a_ia_j\ne1$ for all $i,j$. For nonzero $s\ne t$ with $1-sa_j\ne0$ and $1-ta_j\ne0$ for all $j$,
\[
 \sum_{j=1}^n\frac{(D_A)_{jj}}{(1-sa_j)(1-a_jt)}=\frac{\phi_A(s)-\phi_A(t)}{s-t},
\]
and for $s=t$ the right side is replaced by $\phi_A'(s)$.
\end{lemma}

\begin{proof}
The rational function $\chi(z)=q_A(z)/\bigl(p_A(z)(1-sz)(1-tz)\bigr)$ has a numerator of degree $n$ and a denominator of degree $n+2$, so $\chi(z)=O(z^{-2})$ at infinity and the sum of its residues is zero. Its residue at $a_j$ is the $j$th term of the sum. Its residue at $1/s$ is
\[
 \frac{q_A(1/s)}{p_A(1/s)\,(-s)(1-t/s)}=\frac{\phi_A(s)}{t-s},
\]
because $q_A(1/s)/p_A(1/s)=p_A(s)/q_A(s)$, and its residue at $1/t$ is $\phi_A(t)/(s-t)$ in the same way. This gives the formula for $s\ne t$, and the case $s=t$ follows by letting $t\to s$.
\end{proof}

\begin{corollary}\label{cor-cauchy-inverse}
Let the node sets $a$ and $b$ be simple, nonzero, nonresonant, and disjoint, and let $\Phi$, $D_A$, $D_B$ be formed from them as above. Then $G_A^{-1}=D_AG_AD_A$, so $D_AG_A$ is an involution, and
\[
 \Omega_A=D_AG_A\Phi
 \qquad\text{and}\qquad
 \Omega_B=D_BG_B\Psi,\quad \Psi=\operatorname{diag}\bigl(\phi_A(b_1),\dots,\phi_A(b_n)\bigr),
\]
satisfy $\Omega_A\,(G_A^{-1}H)=-(G_A^{-1}H)\,\Omega_B$. In particular $\Omega_B$ is similar to $-\Omega_A$.
\end{corollary}

\begin{proof}
With $s=a_i$ and $t=a_k$, Lemma~\ref{lem-residue} gives $(G_AD_AG_A)_{ik}=0$ for $i\ne k$, since $\phi_A(a_i)=0$, and $(G_AD_AG_A)_{ii}=\phi_A'(a_i)=p_A'(a_i)/q_A(a_i)$. Hence $G_AD_AG_A=D_A^{-1}$, which is the first claim. For the second, $G_A\Omega_A=G_AD_AG_A\Phi=D_A^{-1}\Phi$, so multiplying the claimed identity on the left by $G_A$ turns it into
\[
 D_A^{-1}\Phi\,G_A^{-1}H=-H\,D_BG_B\Psi ,
\]
and with $G_A^{-1}=D_AG_AD_A$ this becomes $\Phi\,G_AD_AH=-HD_BG_B\Psi$. Lemma~\ref{lem-residue} with $s=a_i$ and $t=b_k$ gives $(G_AD_AH)_{ik}=\phi_A(b_k)/(b_k-a_i)$, and the same lemma for the nodes $b$, with $s=a_i$ and $t=b_k$, gives $(HD_BG_B)_{ik}=-\phi_B(a_i)/(b_k-a_i)$. Both sides therefore have $(i,k)$ entry $\phi_B(a_i)\phi_A(b_k)/(b_k-a_i)$. The matrix $G_A^{-1}H$ is nonsingular by \eqref{eq-general-cauchy-det}, which gives the last claim.
\end{proof}

The matrix $\Omega_A$ carries the angles. In the coordinates of Section~\ref{sec-cauchy}, $P^{-1}FPF^T=S^{-T}\,\widehat P^{-1}\Phi\widehat P\Phi\,S^T$ and $\widehat P^{-1}\Phi\widehat P\Phi=D_x^{-1}G_A^{-1}\Phi G_A\Phi D_x$. By Corollary~\ref{cor-cauchy-inverse} and the commutativity of diagonal matrices,
\begin{equation}\label{eq-Omega-square}
 G_A^{-1}\Phi G_A\Phi=D_AG_AD_A\Phi G_A\Phi=(D_AG_A\Phi)^2=\Omega_A^2,
\end{equation}
so by the proof of Theorem~\ref{thm-main} both matrices of the conjecture are similar to $I-\Omega_A^2$ on the simple-spectrum set. It remains to identify $\Omega_A$ with the matrix $T$ of \eqref{eq-T}. Let $\mathcal V=[a_j^{\,k}]_{j=1,\dots,n;\,k=0,\dots,n-1}$ be the Vandermonde matrix, which maps the coefficient vector of a polynomial of degree less than $n$ to its values at the eigenvalues of $A$, and let $D_{p'}=\operatorname{diag}(p_A'(a_j))$ and $D_q=\operatorname{diag}(q_A(a_j))$, so that $D_A=D_qD_{p'}^{-1}$.

\begin{lemma}\label{lem-vandermonde}
Under (S1), $\mathcal VZ_A=\Lambda\mathcal V$ and $\mathcal VJ\mathcal V^{-1}=D_qG_AD_{p'}^{-1}$. Consequently
\[
 \Omega_A=D_{p'}^{-1}\mathcal V\,T\,\mathcal V^{-1}D_{p'} .
\]
\end{lemma}

\begin{proof}
If $u$ is the coefficient vector of a polynomial $u(z)$, then $Z_Au$ is the coefficient vector of $zu(z)$ modulo $p_A$, whose value at $a_j$ is $a_ju(a_j)$. This is the first identity. For the second, $J$ reverses coefficients, so $(\mathcal VJu)_j=a_j^{n-1}u(1/a_j)$. By Lagrange interpolation $u(z)=\sum_iu(a_i)\,p_A(z)/\bigl((z-a_i)p_A'(a_i)\bigr)$, hence
\[
 a_j^{n-1}u(1/a_j)=\sum_i\frac{a_j^np_A(1/a_j)}{(1-a_ia_j)\,p_A'(a_i)}\,u(a_i)=\sum_i\frac{q_A(a_j)}{(1-a_ja_i)\,p_A'(a_i)}\,u(a_i),
\]
which is the second identity. The first identity gives $\mathcal V\phi_B(Z_A)\mathcal V^{-1}=\phi_B(\Lambda)=\Phi$, so
\[
 \mathcal VT\mathcal V^{-1}=D_qG_AD_{p'}^{-1}\Phi ,
\]
and conjugating by $D_{p'}$ gives $D_{p'}^{-1}D_qG_A\Phi=D_AG_A\Phi=\Omega_A$.
\end{proof}

Since $J^2=I$, Lemma~\ref{lem-vandermonde} also shows directly that $D_AG_A$ is an involution.

\begin{proof}[Proof of Theorem~\ref{thm-normal}, first part]
On the simple-spectrum set, both matrices of the conjecture are similar to $I-\Omega_A^2$ by \eqref{eq-Omega-square}, and $\Omega_A$ is similar to $T$ by Lemma~\ref{lem-vandermonde}. The characteristic polynomials of the two matrices and of $I-T^2$ depend continuously on the data in $\mathcal A$, the last because $T$ is a rational function of the coefficients of $p_A$ and $p_B$ with nonvanishing denominator $\det q_B(Z_A)$. By Lemma~\ref{lem-density} they agree everywhere on $\mathcal A$.
\end{proof}

Thus the angles can be computed from the two characteristic polynomials alone. For the subspace angles between autoregressive moving-average models and the associated cepstral distances \cite{DeCock2002SCL,DeCock2002PhD}, this replaces three Stein solves by a single $n\times n$ eigenvalue problem. Corollary~\ref{cor-cauchy-inverse} also gives a signed form of pole--zero duality. Exchanging the roles of $a$ and $b$ replaces $T$ by $J\phi_A(Z_B)$, and on the simple-spectrum set $\Omega_B$ is similar to $-\Omega_A$. By continuity the characteristic polynomials of $J\phi_B(Z_A)$ and $-J\phi_A(Z_B)$ agree on all of $\mathcal A$, so exchanging the poles and zeros negates the signed sines.

To explain the signs we need a symmetric form for which $T$ is self-adjoint. For a polynomial $g$ of degree less than $n$ write $g(z)/p_A(z)=\sum_{k\ge0}s_kz^{-k-1}$ near infinity, and let $H_n(g/p_A)=[s_{j+k}]_{j,k=0}^{n-1}$ be the associated Hankel matrix of Markov parameters. Let $e_n$ be the last standard basis vector.

\begin{proposition}\label{prop-hankel}
Let $\mathcal G$ be the solution of the Stein equation $\mathcal G-Z_A^T\mathcal GZ_A=e_ne_n^T$, and let $g$ be the polynomial of degree less than $n$ with $g\equiv p_B\,(q_Aq_B)^{-1}\pmod{p_A}$. Then $\mathcal G$ is real symmetric, and
\[
 \mathcal G\,T=H_n(g/p_A),
\]
so $T$ is self-adjoint for the symmetric form defined by $\mathcal G$.
\end{proposition}

\begin{proof}
The Stein equation is uniquely solvable by nonresonance, since $Z_A$ has the eigenvalues of $A$, and its solution is real and symmetric. Both sides of the claimed identity are continuous on $\mathcal A$, so it suffices to verify it under (S1). Put $\widetilde{\mathcal G}=\mathcal V^TD_{p'}^{-1}G_AD_{p'}^{-1}\mathcal V$. By $\mathcal VZ_A=\Lambda\mathcal V$ and $G_A-\Lambda G_A\Lambda=\mathbf 1\mathbf 1^T$, where $\mathbf 1$ is the vector of ones,
\[
 \widetilde{\mathcal G}-Z_A^T\widetilde{\mathcal G}Z_A=\mathcal V^TD_{p'}^{-1}\mathbf 1\mathbf 1^TD_{p'}^{-1}\mathcal V=hh^T,
 \qquad h_k=\sum_j\frac{a_j^k}{p_A'(a_j)} .
\]
By the residue theorem applied to $z^k/p_A(z)$, $h_k=0$ for $k<n-1$ and $h_{n-1}=1$, so $h=e_n$ and $\widetilde{\mathcal G}=\mathcal G$. Lemma~\ref{lem-vandermonde} and $G_A\Omega_A=D_A^{-1}\Phi$ then give
\[
 \mathcal GT=\mathcal V^TD_{p'}^{-1}G_A\Omega_AD_{p'}^{-1}\mathcal V=\mathcal V^T\operatorname{diag}\Bigl(\frac{\phi_B(a_j)}{q_A(a_j)\,p_A'(a_j)}\Bigr)\mathcal V=\mathcal V^T\operatorname{diag}\Bigl(\frac{g(a_j)}{p_A'(a_j)}\Bigr)\mathcal V .
\]
The $(k,l)$ entry of the right side is $\sum_jg(a_j)a_j^{k+l}/p_A'(a_j)=s_{k+l}$, by the residue theorem applied to $z^{k+l}g(z)/p_A(z)$.
\end{proof}

\begin{proof}[Proof of Theorem~\ref{thm-normal}, second part]
If $A$ is Schur stable, then $\mathcal G=\sum_{k\ge0}(Z_A^T)^ke_ne_n^TZ_A^k$ is positive definite, because the pair $(Z_A^T,e_n)$ is controllable. Hence $T=\mathcal G^{-1}(\mathcal GT)$ is similar to the symmetric matrix $\mathcal G^{-1/2}H_n(g/p_A)\mathcal G^{-1/2}$, its eigenvalues are real, and their inertia is that of $H_n(g/p_A)$. By the theorem of Hermite and Hurwitz \cite[Ch.~XV]{Gantmacher1959}, the signature of $H_n(g/p_A)$ is the Cauchy index of $g/p_A$ on the real line. The Cauchy index is a sum of local contributions at the real roots of $p_A$, and at a root $a$ the contribution of $h/p_A$, for $h$ analytic and nonzero at $a$, depends only on the sign of $h(a)$ and on the multiplicity of $a$. The function $g/p_A$ differs from $p_B/(q_Aq_Bp_A)$ by a function analytic at every root of $p_A$, because $g\,q_Aq_B-p_B$ is divisible by $p_A$. If $B$ is also stable, then $q_A(a)q_B(a)>0$ at every real eigenvalue $a$ of $A$, since the factors of $q_A$ and $q_B$ are either positive or come in conjugate pairs. Hence $g/p_A$ and $p_B/p_A$ have the same local contributions and the same Cauchy index.
\end{proof}

When all poles and zeros are real, Theorem~\ref{thm-normal} says that the signs of the sines depend only on how the poles and zeros interlace on the real line. Off the stable set, Proposition~\ref{prop-hankel} still computes the signature of the symmetric matrix $\mathcal GT$, which in indefinite linear algebra is, for semisimple real eigenvalues, the sum over the real eigenvalues of $T$ of their signs weighted by the sign characteristic \cite{GohbergLancasterRodman2005}.

\begin{example}[$n=1$]
With $A=a$ and $B=b$, the matrix $T$ is the scalar $(a-b)/(1-ab)$, and the single eigenvalue of \eqref{eq-main-spectrum} is
\[
 \frac{(1-a^2)(1-b^2)}{(1-ab)^2}=1-\Bigl(\frac{a-b}{1-ab}\Bigr)^2 .
\]
When $a$ and $b$ lie on the same side of the unit circle this is the squared cosine of an ordinary angle whose sine is the pseudo-hyperbolic distance between $a$ and $b$, or between $1/a$ and $1/b$. When they lie on opposite sides it is negative, a hyperbolic angle between a positive and a negative direction.
\end{example}

\begin{example}[$n=2$]\label{ex-n2}
With $\phi_i=\phi_B(a_i)$, a direct computation of the trace and determinant of $\Omega_A$ shows that the signed sines are the roots of $\lambda^2-t\lambda-\phi_1\phi_2$ with
\[
 t=\frac{(1-a_1a_2)(\phi_1-\phi_2)}{a_1-a_2}.
\]
If all four eigenvalues are real, the signed sines are real when $\phi_1\phi_2>0$, and when $\phi_1\phi_2<0$ they form a complex pair exactly when $t^2<4|\phi_1\phi_2|$. By Theorem~\ref{thm-inertia}(b) the latter can happen only when $\kappa_A=\kappa_B=1$. An imaginary $\lambda$ would require $t=0$ and $\phi_1\phi_2<0$ simultaneously, which is impossible. This is the case $n=2$ of Proposition~\ref{prop-key} below.
\end{example}

\section{Angles in an indefinite metric}\label{sec-inertia}

We now prove Theorem~\ref{thm-inertia}. The first ingredient is the inertia of a Stein solution.

\begin{lemma}[Inertia of Stein solutions]\label{lem-inertia}
Let $M\in\R^{n\times n}$ have no eigenvalue on the unit circle and no pair of eigenvalues with product one, let $c\in\R^n$ with $(M,c)$ controllable, and let $X$ solve $X-MXM^T=cc^T$. Then $X$ has exactly as many negative eigenvalues as $M$ has eigenvalues outside the unit circle, and no zero eigenvalue.
\end{lemma}

\begin{proof}
After a real similarity, which changes $X$ by a congruence, $M=M_1\oplus M_2$ with $\spec(M_1)$ inside and $\spec(M_2)$ outside the unit circle. Partition $c=(c_1,c_2)$ and $X$ accordingly. The diagonal blocks solve $X_{ii}-M_iX_{ii}M_i^T=c_ic_i^T$, so
\[
 X_{11}=\sum_{k\ge0}M_1^kc_1c_1^T(M_1^T)^k,\qquad X_{22}=-\sum_{k\ge1}M_2^{-k}c_2c_2^T(M_2^{-T})^{k},
\]
as one checks by substitution. Controllability of $(M,c)$ implies controllability of each $(M_i,c_i)$, hence $X_{11}$ is positive definite and $X_{22}$ is negative definite. A symmetric matrix that is positive definite on a coordinate subspace of dimension $p$ and negative definite on a complementary one of dimension $n-p$ has exactly $p$ positive and $n-p$ negative eigenvalues by the minimax principle.
\end{proof}

By Proposition~\ref{prop:disjoint} the pairs $(A,v)$ and $(B^T,w)$ are controllable, so Lemma~\ref{lem-inertia} gives $\nu(P)=\kappa_A$ and $\nu(Q)=\kappa_B$, where $\nu$ counts negative eigenvalues. In the stable case $P$ and $Q$ are the Gram matrices of the controllability and observability subspaces, and off the stable set they are Gram matrices for an indefinite inner product with $\kappa_A$ and $\kappa_B$ negative squares. In the language of reproducing kernel Pontryagin spaces \cite{AlpayDijksmaRovnyakDeSnoo1997}, the function $\phi_A$ is then a quotient of finite Blaschke products and Lemma~\ref{lem-inertia} is the matrix form of the statement that the associated kernel has $\kappa_A$ negative squares. The second ingredient is the following.

\begin{proposition}\label{prop-key}
Under the hypotheses of Theorem~\ref{thm-main}, the matrix $T$ of \eqref{eq-T} has no eigenvalue on the imaginary axis.
\end{proposition}

\begin{proof}
The eigenvalue $0$ is excluded because $T$ is nonsingular. We identify coefficient vectors in $\C^n$ with polynomials of degree less than $n$. Then $J$ sends $r$ to its reversal $\tilde r(z)=z^{n-1}r(1/z)$, and $\phi_B(Z_A)r$ is the unique polynomial $s$ of degree less than $n$ with $q_Bs\equiv p_Br\pmod{p_A}$. Suppose $Tr=i\theta\, r$ with $\theta\in\R\setminus\{0\}$ and $r\neq0$. Since $J^2=I$, this says $\phi_B(Z_A)r=i\theta\,\tilde r$, that is
\[
 E_1:=i\theta\,q_B\,\tilde r-p_B\,r\equiv0\pmod{p_A}.
\]
Let $h^*$ denote the polynomial with conjugated coefficients. Because $p_A$, $p_B$, and $q_B$ have real coefficients, also $E_2:=E_1^*=-i\theta\,q_B\,\tilde r^*-p_B\,r^*\equiv0$. The combination $\tilde r^*E_1+\tilde rE_2$ eliminates $\theta$ and gives
\[
 p_A\ \big|\ p_B\,s,\qquad s:=\tfrac12\bigl(r\,\tilde r^*+r^*\tilde r\bigr).
\]
The spectra of $A$ and $B$ are disjoint, so $p_A$ divides $s$.

The polynomial $s$ has real coefficients and is palindromic of formal degree $2n-2$, that is $z^{2n-2}s(1/z)=s(z)$, because reversal exchanges $r$ with $\tilde r$ and $r^*$ with $\tilde r^*$. Let $m$ be the multiplicity of $0$ as an eigenvalue of $A$. Then $z^m$ divides $s$, and palindromy gives $\deg s\le2n-2-m$. Writing $s=p_Au$ with $\deg u\le n-2-m$ and reversing gives $s=q_A\cdot z^{n-2}u(1/z)$, so $q_A$ divides $s$ as well. The roots of $q_A$ are the reciprocals of the nonzero eigenvalues of $A$, so nonresonance makes $p_A$ and $q_A$ coprime, and $p_Aq_A$ divides $s$. Since $\deg(p_Aq_A)=2n-m>\deg s$, we conclude $s=0$.

Finally, for $|z|=1$ we have $\tilde r(z)=z^{n-1}\overline{r^*(z)}$ and $\tilde r^*(z)=z^{n-1}\overline{r(z)}$, so
\[
 0=2s(z)=z^{n-1}\bigl(|r(z)|^2+|r^*(z)|^2\bigr)\qquad(|z|=1).
\]
Hence $r=0$, a contradiction.
\end{proof}

On the simple-spectrum set, Proposition~\ref{prop-key} says that the Cauchy pencil $G_A-\lambda\,(D_A\Phi)^{-1}$ has no eigenvalues on the imaginary axis.

\begin{proof}[Proof of Theorem~\ref{thm-inertia}]
\emph{(a)} By Theorem~\ref{thm-normal}, the eigenvalues of \eqref{eq-main-spectrum} are $1-\lambda^2$ with $\lambda\in\spec(T)$. If $1-\lambda^2$ is real, then $\lambda$ is real or purely imaginary, and Proposition~\ref{prop-key} rules out the second case and $\lambda=0$. Hence $\lambda^2>0$ and $1-\lambda^2<1$. The value $0$ is excluded because $R$ is nonsingular by Theorem~\ref{thm-main}.

\emph{(b)} By \eqref{eq-main-spectrum} it suffices to count nonreal eigenvalues of $PQ$. The matrix $PQ$ is self-adjoint for the symmetric forms defined by $Q$ and by $P^{-1}$, since $Q(PQ)$ and $P^{-1}(PQ)$ are symmetric. For a matrix self-adjoint with respect to a nonsingular symmetric $X$, the root subspaces of the eigenvalues in the open upper half-plane span an $X$-neutral subspace, whose dimension is at most $\min\{\pi(X),\nu(X)\}$ \cite{GohbergLancasterRodman2005}. With $\nu(P^{-1})=\kappa_A$ and $\nu(Q)=\kappa_B$ this gives (b).

\emph{(c), (d)} The eigenvalues of \eqref{eq-main-spectrum} are $1/(1+\mu)$ with $\mu\in\spec(PQ)$, so negative eigenvalues correspond to real $\mu<-1$ and eigenvalues in $(0,1)$ to real $\mu>0$, and by (a) no real $\mu$ lies in $[-1,0]$. Let $\nu(t)$ be the number of negative eigenvalues of the symmetric matrix $P-tQ^{-1}$. Since $P-tQ^{-1}=(PQ-tI)Q^{-1}$, this matrix is singular exactly when $t\in\spec(PQ)$, and $\det(P-tQ^{-1})$ changes sign at such $t$ exactly when the algebraic multiplicity is odd. Hence $\nu$ is constant between real eigenvalues of $PQ$, and at an eigenvalue of multiplicity $m$ it jumps by at most $m$ and by an amount congruent to $m$ modulo $2$. Now $\nu(0)=\kappa_A$, $\nu(t)=\nu(Q)=\kappa_B$ for $t\to-\infty$, and $\nu(t)=n-\kappa_B$ for $t\to+\infty$. Summing the jumps over $(-\infty,0)$ and over $(0,\infty)$ gives (c) and (d).

The final statement follows from (b), since $N_c=0$ when $\kappa_A\in\{0,n\}$ or $\kappa_B\in\{0,n\}$.
\end{proof}

\section{Reflection across the unit circle}\label{sec-reflection}

Theorem~\ref{thm-inertia} counts the eigenvalues of \eqref{eq-main-spectrum} but does not say how they are related to the stable case. This section gives that relation. Reflecting a pole $a$ to $1/a$, or a zero $b$ to $1/b$, changes the spectral density of the model only by a constant factor, provided nonreal roots are reflected together with their conjugates. A nonresonant model with simple nonzero poles and zeros therefore has the same spectral density, up to a constant, as a stable minimum-phase model, and we show that the signed sines of the two are related by a principal pivot transform.

On the simple-spectrum set the signed sines are the eigenvalues of $\Omega_A$, by Lemma~\ref{lem-vandermonde}. To display the dependence on the eigenvalues we write $\Omega(a,b)=D_AG_A\Phi$ for node sets $a=(a_1,\dots,a_n)$ and $b=(b_1,\dots,b_n)$, so that $\Omega_A=\Omega(a,b)$ and $\Omega_B=\Omega(b,a)$.

For an index set $\alpha\subset\{1,\dots,n\}$ with complement $\bar\alpha$, and a matrix $X$ whose principal submatrix $X_{\alpha\alpha}$ is nonsingular, the principal pivot transform of $X$ relative to $\alpha$ is
\[
 \ppt_\alpha(X)=\begin{bmatrix}X_{\alpha\alpha}^{-1}&-X_{\alpha\alpha}^{-1}X_{\alpha\bar\alpha}\\ X_{\bar\alpha\alpha}X_{\alpha\alpha}^{-1}&X_{\bar\alpha\bar\alpha}-X_{\bar\alpha\alpha}X_{\alpha\alpha}^{-1}X_{\alpha\bar\alpha}\end{bmatrix}
\]
in the corresponding block ordering. If $y=Xu$, then $\ppt_\alpha(X)$ is the matrix that maps $(y_\alpha,u_{\bar\alpha})$ to $(u_\alpha,y_{\bar\alpha})$, so it exchanges the roles of input and output in the coordinates indexed by $\alpha$ \cite{Tsatsomeros2000}. It is immediate from the definition that $\ppt_\alpha(\Delta X\Delta^{-1})=\Delta\ppt_\alpha(X)\Delta^{-1}$ for every nonsingular diagonal $\Delta$, and that $\det\ppt_\alpha(X)=\det X_{\bar\alpha\bar\alpha}/\det X_{\alpha\alpha}$, because the Schur complement of the pivoted block in $\ppt_\alpha(X)$ is $X_{\bar\alpha\bar\alpha}$. Moreover $\ppt_{\alpha\cup\gamma}(X)=\ppt_\gamma(\ppt_\alpha(X))$ for disjoint $\alpha$ and $\gamma$ whenever both sides are defined \cite{Tsatsomeros2000}.

\begin{lemma}\label{lem-pivot}
Let $a_1,\dots,a_n$ be distinct and nonzero, let $b_1,\dots,b_n$ be such that the data are nonresonant and $a_i\ne b_k$ for all $i,k$, fix an index $i$, and let $a'$ be obtained from $a$ by replacing $a_i$ with $1/a_i$. Then the data $(a',b)$ have the same properties, the diagonal entry $\Omega(a,b)_{ii}$ is nonzero, and
\[
 \ppt_{\{i\}}\bigl(\Omega(a,b)\bigr)=\Delta\,\Omega(a',b)\,\Delta^{-1},
\]
where $\Delta$ is diagonal with diagonal entries $\delta_i=1$ and
\[
 \delta_j=-\frac{(1-a_i^2)\,\phi_B(a_i)\,(1-a_ia_j)}{a_i\,(D_A)_{ii}\,(a_j-a_i)}\qquad(j\ne i).
\]
\end{lemma}

\begin{proof}
The new products of nodes are $a_j/a_i$, $1/a_i^2$, and $b_k/a_i$, and none of them equals one because the $a_j$ are distinct, $a_i^2\ne1$, and $b_k\ne a_i$. Also $1/a_i\ne a_j$ because $a_ia_j\ne1$, and $1/a_i\ne b_k$ because $a_ib_k\ne1$. The diagonal entry $\Omega(a,b)_{ii}=(D_A)_{ii}\phi_B(a_i)/(1-a_i^2)$ is nonzero because $q_A(a_i)\ne0$ and $p_B(a_i)\ne0$.

Take $i=1$ for notational simplicity, write $d_j=(D_A)_{jj}$, $\phi_j=\phi_B(a_j)$, and $\tau_j=(a_j-a_1)/(1-a_1a_j)$, and use primes for the quantities built from $a'$. Since $p_B(1/a)=a^{-n}q_B(a)$ and $q_B(1/a)=a^{-n}p_B(a)$, we have $\phi'_1=1/\phi_1$ and $\phi'_j=\phi_j$ for $j\ge2$. The weights factor as $d_j=(1-a_j^2)\prod_{k\ne j}(1-a_ka_j)/(a_j-a_k)$, and comparing factors gives
\[
 d'_1=-\frac{(1-a_1^2)^2}{a_1^2\,d_1},\qquad d'_j=\tau_j^2\,d_j\quad(j\ge2).
\]
Finally $G'_{11}=-a_1^2/(1-a_1^2)$, $G'_{1j}=G'_{j1}=a_1/(a_1-a_j)$, and $G'_{jk}=(G_A)_{jk}$ for $j,k\ge2$.

Write $\Omega=\Omega(a,b)$ and $\Omega'=\Omega(a',b)$, and put $\eta=a_1d_1/\bigl((1-a_1^2)\phi_1\bigr)$, so that $\delta_j=-1/(\eta\tau_j)$. For the pivot entry, $\Omega'_{11}=d'_1G'_{11}\phi'_1=(1-a_1^2)/(d_1\phi_1)=1/\Omega_{11}$. In the first row, $-\Omega_{1j}/\Omega_{11}=-(1-a_1^2)\phi_j/\bigl((1-a_1a_j)\phi_1\bigr)$ and $\Omega'_{1j}=-(1-a_1^2)^2\phi_j/\bigl(a_1d_1(a_1-a_j)\bigr)$, whose ratio is $-\eta\tau_j=\delta_1/\delta_j$. In the first column, $\Omega_{j1}/\Omega_{11}=d_j(1-a_1^2)/\bigl((1-a_1a_j)d_1\bigr)$ and $\Omega'_{j1}=a_1\tau_j^2d_j/\bigl((a_1-a_j)\phi_1\bigr)$, whose ratio is $-1/(\eta\tau_j)=\delta_j/\delta_1$. For $j,k\ge2$ we use the identity
\[
\frac1{1-xy}-\frac{1-c^2}{(1-xc)(1-cy)}=\frac{(x-c)(y-c)}{(1-xy)(1-xc)(1-cy)},
\]
which follows by expanding the numerator $(1-xc)(1-cy)-(1-c^2)(1-xy)=(x-c)(y-c)$. With $x=a_j$, $y=a_k$, and $c=a_1$ it gives
\[
\Omega_{jk}-\frac{\Omega_{j1}\Omega_{1k}}{\Omega_{11}}=d_j\phi_k\Bigl(\frac1{1-a_ja_k}-\frac{1-a_1^2}{(1-a_ja_1)(1-a_1a_k)}\Bigr)=\frac{\tau_j\tau_k\,d_j\phi_k}{1-a_ja_k}=\frac{\tau_k}{\tau_j}\,\Omega'_{jk},
\]
and $\tau_k/\tau_j=\delta_j/\delta_k$. Every entry of $\ppt_{\{1\}}(\Omega)$ therefore equals $\delta_j/\delta_k$ times the corresponding entry of $\Omega'$, which is the claim.
\end{proof}

Reflecting several poles one at a time keeps the data nonresonant at every step, so by Lemma~\ref{lem-pivot} each successive pivot entry is nonzero. By the determinant formula, $\det\Omega_{\alpha\cup\{i\},\alpha\cup\{i\}}=\det\Omega_{\alpha\alpha}\cdot\ppt_\alpha(\Omega)_{ii}$, so induction shows that $\Omega(a,b)_{\alpha\alpha}$ is nonsingular for every $\alpha$, and the composition rule for principal pivot transforms gives the first part of the following theorem. For the zeros we use Corollary~\ref{cor-cauchy-inverse}, which says that $\Omega(b,a)$ is similar to $-\Omega(a,b)$.

\begin{theorem}[Reflection principle]\label{thm-reflection}
Let the node sets $a$ and $b$ be simple, nonzero, nonresonant, and disjoint. Let $\alpha$ be a set of indices of the $a_i$, and let $a^\alpha$ be obtained from $a$ by replacing $a_i$ with $1/a_i$ for $i\in\alpha$. Then
\[
 \Omega(a^\alpha,b)\ \text{is similar to}\ \ppt_\alpha\bigl(\Omega(a,b)\bigr).
\]
Similarly, if $\gamma$ is a set of indices of the $b_k$ and $b^\gamma$ is defined in the same way, then $\Omega(a,b^\gamma)$ is similar to $-\ppt_\gamma\bigl(\Omega(b,a)\bigr)$.
\end{theorem}

Every such pair of node sets is realized by data in the domain of Theorem~\ref{thm-main}. Take $A$ in companion form with $v=e_1$ and choose $w$ so that $A+vw^T$ has characteristic polynomial $p_B$. The pair $(A,v)$ is controllable, and $(B^T,w)$ is controllable because $p_A$ and $p_B$ are coprime, so $P$ and $Q$ are nonsingular by Lemma~\ref{lem-inertia}. When the data are real, a reflected set should contain each nonreal eigenvalue together with its conjugate, so that the reflected data are real again. The theorem itself does not need this, since Lemma~\ref{lem-pivot} holds for complex nodes.

Applied with $a$ and $b$ the poles and zeros of the stable minimum-phase model of a spectral density, Theorem~\ref{thm-reflection} describes every other nonresonant realization with simple nonzero spectra and the same density up to a constant. The canonical correlations of the stationary process, with their signs, are the eigenvalues of $\Omega(a,b)$. Moving a set of poles outside the unit circle replaces $\Omega(a,b)$ by its principal pivot transform on those poles, and moving a set of zeros outside does the same in the dual coordinates. Heuristically, reflecting a root exchanges input and output in the coordinate of that root, up to the diagonal scaling of Lemma~\ref{lem-pivot}, and this matches the fact that reflection turns the corresponding mode from causal to anticausal, exchanging its past and future. We do not make this reading precise here. Two cases are completely explicit.

\begin{corollary}\label{cor-allreflect}
Let $\rho_1,\dots,\rho_n$ be the past and future canonical correlations of a stable minimum-phase model with simple nonzero poles and zeros. If all poles are reflected and the zeros are kept, or all zeros are reflected and the poles are kept, the eigenvalues of \eqref{eq-main-spectrum} for the reflected data are $1-1/\rho_k^2$, $k=1,\dots,n$. If all poles and all zeros are reflected, they are $1-\rho_k^2$, the same as for the stable model.
\end{corollary}

\begin{proof}
For $\alpha=\{1,\dots,n\}$ the principal pivot transform is the inverse, so $\Omega(1/a,b)$ is similar to $\Omega(a,b)^{-1}$ and its eigenvalues are the reciprocals of the signed canonical correlations. For the zeros, $\Omega(a,1/b)$ is similar to $-\Omega(b,a)^{-1}$, and $\Omega(b,a)$ is similar to $-\Omega(a,b)$, which again gives reciprocals. Reflecting both inverts twice. The eigenvalues of \eqref{eq-main-spectrum} are $1-\lambda^2$ for $\lambda$ in these spectra, by Theorem~\ref{thm-normal}.
\end{proof}

For $n=1$ Corollary~\ref{cor-allreflect} is the identity $(a-b)^2+(1-a^2)(1-b^2)=(1-ab)^2$, which shows that the pseudo-hyperbolic distances of $(a,b)$ and $(1/a,b)$ are reciprocal. In general the principal pivot transform depends on more than the spectrum of $\Omega(a,b)$, which is why intermediate reflections can produce complex eigenvalues even though the stable spectrum is real. The determinant formula gives a product rule. After reflecting the poles indexed by $\alpha$, the product of the signed sines equals $\det\Omega(a,b)_{\bar\alpha\bar\alpha}/\det\Omega(a,b)_{\alpha\alpha}$, a ratio of principal minors of the stable matrix.

\section{Several channels}\label{sec-mimo}

De Cock and De Moor record that the conjecture fails when the rank-one perturbation is replaced by one of higher rank \cite{DeCock2004}. In this section we explain the failure and identify what survives. Throughout this section $A$ and $B$ are Schur stable, so the angles are genuine.

Let $V,W\in\R^{n\times m}$, $B=A+VW^T$, and let $P,Q,R$ solve \eqref{eq-stein} with $v,w$ replaced by $V,W$. Assume that $(A,V)$ is controllable and $(W^T,B)$ is observable, so that $P$ and $Q$ are positive definite. In the innovation model $x(t+1)=Ax(t)+V\varepsilon(t)$, $y(t)=-W^Tx(t)+\varepsilon(t)$, driven by white noise $\varepsilon$ with identity covariance, the inverse system has state matrix $B$, and the nonzero squared canonical correlations between past and future outputs are the eigenvalues of $PQ(I+PQ)^{-1}$ \cite{DeCock2002PhD,LindquistPicci2015}. For $M\in\R^{n\times n}$ and $X\in\R^{n\times p}$ write
\[
\Reach(M,X)=\bigl\{\,k\mapsto u^TM^kX\;:\;u\in\R^n\bigr\}\subset\ell^2\bigl(\mathbb{Z}_{\ge0};\R^{1\times p}\bigr).
\]
When $(M,X)$ is controllable and $M$ is stable this is an $n$-dimensional subspace, and in the basis indexed by $u$ its Gram matrix is the solution of $Y-MYM^T=XX^T$. For two such pairs $(M_1,X_1)$ and $(M_2,X_2)$ with $X_1,X_2$ of the same width, the cross Gram matrix $Y_{12}$ solves $Y_{12}-M_1Y_{12}M_2^T=X_1X_2^T$, and the squared cosines of the principal angles are the eigenvalues of $Y_1^{-1}Y_{12}Y_2^{-1}Y_{12}^T$ \cite{BjorckGolub1973}. In this language the left side of the conjecture concerns the pair $\Reach(A,V)$ and $\Reach(B^T,W)$, with Gram matrices $P$ and $Q$ and cross Gram matrix $R$.

\begin{remark}[What the left side measures]
Drive the two filters $x(t+1)=Ax(t)+V\varepsilon(t)$ and $z(t+1)=B^Tz(t)+W\varepsilon(t)$ by the same noise. In steady state $P$, $Q$, and $R$ are the covariance of $x$, the covariance of $z$, and their cross covariance, so the eigenvalues of $P^{-1}RQ^{-1}R^T$ are the squared canonical correlations between the two state vectors. For one channel the two states span $n$ lags of the autoregressive processes $\varepsilon/p_A$ and $\varepsilon/p_B$, and the conjecture says that the past and future canonical correlations of $(p_B/p_A)\varepsilon$ are the complements of the canonical correlations between two autoregressive processes driven by the same noise, which is De Cock's description \cite{DeCock2002PhD}. For several channels the left side is still a correlation between two filter states, but it is no longer a function of the output process alone.
\end{remark}

\subsection{The correct pair of subspaces}

Since $Q$ is positive definite, the columns of $\begin{bmatrix}Q^{1/2}BQ^{-1/2}\\ W^TQ^{-1/2}\end{bmatrix}$ are orthonormal by the second Stein equation, and they can be completed to an orthogonal matrix of size $n+m$. Equivalently there are $B_\perp\in\R^{n\times m}$ and $D_\perp\in\R^{m\times m}$ with
\begin{equation}\label{eq-lossless}
BQ^{-1}B^T+B_\perp B_\perp^T=Q^{-1},\qquad BQ^{-1}W+B_\perp D_\perp^T=0,\qquad W^TQ^{-1}W+D_\perp D_\perp^T=I .
\end{equation}
The completion is unique up to $B_\perp\mapsto B_\perp O$, $D_\perp\mapsto D_\perp O$ with $O$ orthogonal. Define $\check V=VD_\perp-B_\perp$.

\begin{theorem}\label{thm-structure}
The matrices $\check P=P+Q^{-1}$, $\check Q=Q^{-1}$, and $\check R=-Q^{-1}$ satisfy
\[
\check P-A\check PA^T=\check V\check V^T,\qquad
\check Q-B\check QB^T=B_\perp B_\perp^T,\qquad
\check R-A\check RB^T=\check VB_\perp^T .
\]
Consequently they are the Gram matrices and the cross Gram matrix of $\Reach(A,\check V)$ and $\Reach(B,B_\perp)$, and
\[
\spec\bigl(\check P^{-1}\check R\check Q^{-1}\check R^T\bigr)=\spec\bigl((P+Q^{-1})^{-1}Q^{-1}\bigr)=\spec\bigl((I+PQ)^{-1}\bigr).
\]
That is, for every number of channels, the complements of the squared canonical correlations are the squared cosines of the principal angles between $\Reach(A,\check V)$ and $\Reach(B,B_\perp)$.
\end{theorem}

\begin{proof}
The second identity is the first relation in \eqref{eq-lossless}. For the third, write $A=B-VW^T$ and use the first two relations in \eqref{eq-lossless},
\[
Q^{-1}-AQ^{-1}B^T=\bigl(Q^{-1}-BQ^{-1}B^T\bigr)+V\,(BQ^{-1}W)^T=B_\perp B_\perp^T-VD_\perp B_\perp^T=-\check VB_\perp^T .
\]
For the first, expand $AQ^{-1}A^T$ with $A=B-VW^T$ and use all three relations in \eqref{eq-lossless},
\[
AQ^{-1}A^T=Q^{-1}-B_\perp B_\perp^T+B_\perp D_\perp^TV^T+VD_\perp B_\perp^T+V(I-D_\perp D_\perp^T)V^T ,
\]
so that $\check P-A\check PA^T=VV^T+Q^{-1}-AQ^{-1}A^T=(B_\perp-VD_\perp)(B_\perp-VD_\perp)^T$. The pair $(A,\check V)$ is controllable because $\check P$ is positive definite, and $(B,B_\perp)$ is controllable because $Q^{-1}$ is. Finally $(P+Q^{-1})^{-1}Q^{-1}=\bigl(Q(P+Q^{-1})\bigr)^{-1}=(I+QP)^{-1}$, which has the spectrum of $(I+PQ)^{-1}$.
\end{proof}

The matrix $P+Q^{-1}$ is the covariance of the backward predictor state expressed in forward coordinates \cite{LindquistPicci2015}, so Theorem~\ref{thm-structure} says that the canonical correlations are angles between a reachability space of the backward model and a reachability space of the inverse model. For one channel these are the subspaces of the conjecture.

\begin{proposition}\label{prop-siso}
If $M$ is stable and $(M,c)$ is controllable with $c\in\R^n$, then $\Reach(M,c)$ is the space of all square-summable real sequences annihilated by the recurrence with characteristic polynomial $\det(zI-M)$. In particular it depends only on the characteristic polynomial of $M$.
\end{proposition}

\begin{proof}
Each sequence $k\mapsto u^TM^kc$ satisfies the recurrence by the Cayley--Hamilton theorem, so $\Reach(M,c)$ lies in its $n$-dimensional solution space, and controllability gives dimension $n$.
\end{proof}

For $m=1$ it follows that $\Reach(A,\check V)=\Reach(A,v)$ and $\Reach(B,B_\perp)=\Reach(B^T,w)$, since $B$ and $B^T$ have the same characteristic polynomial. Principal angles depend only on the subspaces, so Theorem~\ref{thm-structure} and Proposition~\ref{prop-siso} give a short geometric proof of the stable case of Theorem~\ref{thm-main}. It is close in spirit to De Cock's use of the fact that single-input controllability spaces depend only on the poles \cite[Proposition~3.2]{DeCock2002PhD}, combined with the backward model. For one channel the backward input is also explicit, and this explains the matrix $F$ of Theorem~\ref{thm-main}.

\begin{proposition}\label{prop-backward}
For $m=1$, $\check V=\pm F^{-1}v$, with the sign depending on the choice of completion.
\end{proposition}

\begin{proof}
Since $(A,v)$ is controllable, $\check V=f(A)v$ for a polynomial $f$. The Gramian of $(A,\check V)$ is $f(A)Pf(A)^T$, and it equals $P+Q^{-1}=F^{-1}PF^{-T}$ by Theorem~\ref{thm-structure} and the second identity of \eqref{eq-I1I2}. Hence $K=Ff(A)$ commutes with $A$ and satisfies $KPK^T=P$. The Gramian of $(A,Kv)$ is therefore $P$, and a Gramian $Y$ determines its input through $Y-AYA^T=uu^T$, so $Kvv^TK^T=vv^T$ and $Kv=\pm v$. As $K$ is a polynomial in $A$ and $v$ is cyclic, $K=\pm I$.
\end{proof}

Thus $F$ is the change of state coordinates between the forward and the backward model, and the second identity of \eqref{eq-I1I2} is the transformation law of the Gramian under it.

\subsection{Why several channels fail}

For $m\ge2$ a reachability space depends on the input matrix and not only on the spectrum, and both coincidences above fail. The conjecture then compares the angles of one pair of subspaces with the canonical correlations, which by Theorem~\ref{thm-structure} are the angles of a different pair. This is the mechanism behind the counterexamples recorded in \cite{DeCock2004}. The following statement shows that the two identities of Theorem~\ref{thm-main} capture the conjecture exactly, for any number of channels.

\begin{proposition}\label{prop-equiv}
For any $m$, with $A$ and $B$ stable, the identity $\spec(P^{-1}RQ^{-1}R^T)=\spec\bigl((I+PQ)^{-1}\bigr)$ holds if and only if there is a real matrix $F$ with
\[
RQ^{-1}R^T=P-FPF^T\qquad\text{and}\qquad F(P+Q^{-1})F^T=P .
\]
\end{proposition}

\begin{proof}
If such an $F$ exists, it is nonsingular because $F(P+Q^{-1})F^T=P$ with $P$ positive definite, and the proof of Theorem~\ref{thm-main} applies verbatim. Conversely put $\check P=P+Q^{-1}$. The solutions of the second equation are $F=P^{1/2}O\check P^{-1/2}$ with $O$ orthogonal, and for these the first equation reads
\[
O\bigl(\check P^{-1/2}P\check P^{-1/2}\bigr)O^T=I-P^{-1/2}RQ^{-1}R^TP^{-1/2}.
\]
Both sides are symmetric, so a suitable $O$ exists if and only if they have the same spectrum. The spectrum of the left matrix is that of $\check P^{-1}P=I-(I+QP)^{-1}$, and the spectrum of the right is that of $I-P^{-1}RQ^{-1}R^T$, so the condition is the identity itself.
\end{proof}

For one channel a solution is given in closed form by $F=\phi_B(A)$. For several channels the identity appears, in numerical experiments based on the rank of the Jacobian of the spectral mismatch, to hold on a set of codimension $n$ in the space of input and output matrices, and we know of no structural description of this set beyond decoupled channels.

\subsection{A common cosine--sine structure}

Although the two pairs of subspaces differ, they sit inside the same two larger subspaces. For a monic polynomial $p$ of degree $n$ with roots in the open unit disk, let $K_p\subset\ell^2(\mathbb{Z}_{\ge0};\R)$ be the space of square-summable sequences annihilated by the recurrence with characteristic polynomial $p$. It has dimension $n$, and it is the real model space attached to the roots of $p$.

\begin{lemma}[Channel hulls]\label{lem-hull}
If $M$ is stable with characteristic polynomial $p$ and $X\in\R^{n\times m}$, then $\Reach(M,X)\subset K_p\otimes\R^{1\times m}$.
\end{lemma}

\begin{proof}
Each channel $k\mapsto u^TM^kx_l$ is annihilated by the recurrence of $p$, by the Cayley--Hamilton theorem.
\end{proof}

Write $\mathcal H_A=K_{p_A}\otimes\R^{1\times m}$ and $\mathcal H_B=K_{p_B}\otimes\R^{1\times m}$, each of dimension $nm$. The orthogonal projections onto them are $P_{K_{p_A}}\otimes I$ and $P_{K_{p_B}}\otimes I$, so the principal angles between $\mathcal H_A$ and $\mathcal H_B$ are the principal angles $\theta_1\le\dots\le\theta_n$ between $K_{p_A}$ and $K_{p_B}$, each repeated $m$ times. These scalar angles depend only on the poles and zeros, and their sines are the absolute values of the eigenvalues of $T$ by Theorem~\ref{thm-normal}. By Lemma~\ref{lem-hull} both the pair of the conjecture and the pair of Theorem~\ref{thm-structure} consist of an $n$-dimensional subspace of $\mathcal H_A$ and an $n$-dimensional subspace of $\mathcal H_B$. If $E$ and $E'$ are orthonormal frames for two such subspaces, written in the bases $\{u_j\otimes e_l\}$ and $\{u'_j\otimes e_l\}$ built from principal vectors $u_j\in K_{p_A}$ and $u'_j\in K_{p_B}$ of the scalar pair, then the squared cosines of their principal angles are the squared singular values of
\[
E^T\bigl(\operatorname{diag}(\cos\theta_1,\dots,\cos\theta_n)\otimes I_m\bigr)E' .
\]
For several channels the conjecture and the canonical correlations are therefore two different compressions of the same block cosine matrix.

\begin{theorem}[Interlacing]\label{thm-interlace}
Let $c_1\ge c_2\ge\dots\ge c_{nm}$ be the numbers $\cos^2\theta_j$, $j=1,\dots,n$, each listed $m$ times. If $\mathcal S\subset\mathcal H_A$ and $\mathcal S'\subset\mathcal H_B$ are subspaces of dimension $n$ and $\gamma_1\ge\dots\ge\gamma_n$ are the squared cosines of the principal angles between them, then $\gamma_k\le c_k$ for $k=1,\dots,n$. In particular the eigenvalues of $(I+PQ)^{-1}$ and those of $P^{-1}RQ^{-1}R^T$, ordered decreasingly, are bounded entrywise by $c_1,\dots,c_n$.
\end{theorem}

\begin{proof}
The numbers $\gamma_k$ are the eigenvalues of the compression of $P_{\mathcal S'}$ to $\mathcal S$. Since $\mathcal S'\subset\mathcal H_B$, we have $P_{\mathcal S'}\le P_{\mathcal H_B}$, and by monotonicity of eigenvalues $\gamma_k$ is at most the $k$th eigenvalue of the compression of $P_{\mathcal H_B}$ to $\mathcal S$. That compression is itself the compression to $\mathcal S\subset\mathcal H_A$ of the compression of $P_{\mathcal H_B}$ to $\mathcal H_A$, whose eigenvalues are $c_1,\dots,c_{nm}$. Cauchy interlacing finishes the proof. The last statement follows from Theorem~\ref{thm-structure} and Lemma~\ref{lem-hull}.
\end{proof}

In terms of the canonical correlations $\rho_k$ of the $m$-channel model, Theorem~\ref{thm-interlace} says that the $k$th smallest canonical correlation is at least the $k$th smallest of the scalar pole--zero sines $\sin\theta_j$, each counted $m$ times. For $m=1$ the inequalities are equalities, and in numerical experiments with $m=2$ and $m=3$ the gap between the two sides falls below $10^{-7}$ in some cases.

\section{Discussion}\label{sec-discussion}

The proof identifies a simple algebraic mechanism behind the De Cock--De Moor conjecture. The decisive assumption is the rank-one relation $B-A=vw^T$. On the generic set, this relation forces the change-of-basis matrix between eigenbases of $A$ and $B$ to be a diagonally scaled Cauchy matrix. The Stein equations then produce a second family of Cauchy matrices, and a single rational interpolation identity yields the two basis-free identities \eqref{eq-I1I2}. In the stable case, Proposition~\ref{prop-equiv} shows that the existence of some matrix satisfying these two identities is not only sufficient but equivalent to the conjecture, for any number of channels, and Proposition~\ref{prop-backward} shows that the matrix $F$ in them is the change of coordinates between the forward and backward models.

This point of view also clarifies the role of the hypotheses in the original problem. Stability is natural in the system-theoretic setting, where it permits Gramian and infinite-series representations, but it plays no role in the finite-dimensional argument. The natural algebraic hypothesis is instead nonresonance of the associated Stein operators, and on that domain the spectra of $A$ and $B$ are automatically disjoint by Proposition~\ref{prop:disjoint}. The generic assumptions used in the proof are only a device for exposing the Cauchy structure.

Taking determinants in \eqref{eq-main-similarity} gives
\[
\det(R)^2\det(I+PQ)=\det(P)\det(Q),
\]
which also connects the theorem to the original systems interpretation. In the stable setting of \cite{DeCock2004}, it is equivalent to the known product formula $\prod_i\cos^2\theta_i=\prod_i(1-\rho_i^2)$ relating principal angles and past--future canonical correlations, which underlies the cepstral norm identity of \cite{Martin2000,DeCock2002SCL}; see also \cite{JewellBloomfield1983}. In that setting De Cock obtained it as an expression for the mutual information between past and future \cite[Eq.~(4.23)]{DeCock2002PhD}. Theorem~\ref{thm-main} strengthens this product relation to similarity of the two matrices.

We now compare with the independent work of Gillberg and L\"ofberg \cite{GillbergLofberg2026}, which we read after the Cauchy argument was complete. They construct the explicit intertwiner $\phi_B(A)^{-1}P=U^{-1}$ and prove that it intertwines the two matrices of the conjecture whenever, in addition to nonsingular $P$, $Q$, and $I+PQ$, the spectra of $A$ and $B$ are disjoint, with repeated and defective eigenvalues allowed. Their proof rests on a Lyapunov-kernel identity, which is the counterpart of Lemma~\ref{lem-cauchy-identity}, and on the classical $AB$--$BA$ principle, which is the counterpart of the last step in the proof of Theorem~\ref{thm-main}. The intertwiner is theirs, and it was in comparing with their construction that we recognized the identities of Corollary~\ref{cor-two-identities} as coordinate forms of the basis-free identities \eqref{eq-I1I2}. Proposition~\ref{prop:disjoint} shows that the disjointness hypothesis is automatic on the domain of the conjecture, so their intertwiner is defined there. The remaining results of this paper, the normal form, the inertia theorem, the reflection principle, and the treatment of several channels, do not, to our knowledge, appear in \cite{GillbergLofberg2026}.

The statement of Theorem~\ref{thm-main} has also been verified in the Lean~4 proof assistant with the Mathlib library. The formal proof does not follow the density argument of Section~\ref{sec-proof}. It proves the identities \eqref{eq-I1I2} directly for arbitrary solutions of the Stein equations, using the Cayley--Hamilton theorem and Krylov bases, and then applies the algebraic reduction in the proof of Theorem~\ref{thm-main}. The formalization covers Theorem~\ref{thm-main} for real data but not the remaining results of the paper. The source is provided as supplementary material.

In the stable case the space of functions $u\mapsto u^T(I-zA)^{-1}v$ is the model space $K_{\phi_A}$, the Gramian $P$ is its Gram matrix, and the identities \eqref{eq-I1I2} can be read as coordinate forms of the projection formula for model spaces. Off the stable set the corresponding reproducing kernel spaces are Pontryagin spaces with $\kappa_A$ and $\kappa_B$ negative squares, and Theorem~\ref{thm-inertia} is a statement about angles between two subspaces of an indefinite space. Theorem~\ref{thm-reflection} suggests that these indefinite angles are obtained from the Hilbert space angles of the stable model by exchanging past and future in the coordinates of the reflected roots, in the spirit of the exchange operators that relate contractive and $J$-contractive operators. Making this precise in operator-theoretic terms is a natural next step. For several channels, Theorem~\ref{thm-structure} gives the correct subspaces and Theorem~\ref{thm-interlace} gives an inequality that is nearly attained in numerical experiments, but a structural description of the models for which the original identity holds remains open.

\section*{Declaration of competing interest}

The author declares no known competing financial interests or personal relationships that could have appeared to influence the work reported in this paper.

\section*{Data and code availability}

No data were used for the research described in this article. The supplementary material contains the Lean~4 source of the formal verification of Theorem~\ref{thm-main}, a notebook that builds it, and a Python script that reproduces the numerical checks mentioned in the text.

\section*{Declaration of generative AI and AI-assisted technologies in the writing process}

During the preparation of this work the author used the generative AI assistants Google Gemini, OpenAI ChatGPT, and Anthropic Claude in the development, drafting, and checking of this manuscript, including literature and priority searches, numerical verification, and the preparation of the Lean formalization. These tools were not relied upon as authorities for correctness. After using these tools, the author checked the arguments step by step, reviewed and edited the content as needed, and takes full responsibility for the content of the publication.

\end{document}